\documentclass[12pt]{amsart}
\usepackage{amsmath,amssymb,amsthm}
\usepackage{moreverb}

\usepackage[latin1]{inputenc}
\usepackage[T1]{fontenc}

\newcommand{\ignore}[1]{}

\newtheorem{theorem}{Theorem}[section]
\newtheorem{lemma}[theorem]{Lemma}
\newtheorem{corollary}[theorem]{Corollary}
\newtheorem{proposition}[theorem]{Proposition}

\theoremstyle{definition}
\newtheorem{definition}[theorem]{Definition}
\newtheorem{example}[theorem]{Example}

\theoremstyle{remark}
\newtheorem{remark}[theorem]{Remark}

\numberwithin{equation}{section}

\input xy
\xyoption{all}

\newcommand{\bN}{{\mathbb{N}}}

\newcommand{\bQ}{{\mathbb{Q}}}
\newcommand{\bR}{\mathbb{R}}

\newcommand{\hlt}{\mathrm {ht\, }}

\newcommand{\supp}{\mbox{\rm{supp}} }

\newcommand{\End}{\mathrm{End}}

\newcommand{\Hom}{\mbox{\rm{Hom}} }

\newcommand{\fM}{{\mathfrak m}}

\newcommand{\fa}{\mathfrak{a}}

\newcommand{\cC}{{\mathcal C}}

\newcommand{\cF}{{\mathcal F}}

\newcommand{\lra}{{\longrightarrow}}

\begin{document}

\title[Frobenius and Cartier algebras of Stanley-Reisner rings]{Frobenius and Cartier algebras of Stanley-Reisner rings}

\author[J. \`Alvarez Montaner]{Josep \`Alvarez Montaner$^*$ $^{\dag}$ }
\thanks{$^*$Partially supported by MTM2010-20279-C02-01}
\thanks{$\dag$Partially supported by SGR2009-1284}
\address{Dept. Matem\`atica Aplicada I\\
Univ. Polit\`ecnica de Catalunya\\ Av. Diagonal 647,
Barcelona 08028, SPAIN} \email{Josep.Alvarez@upc.edu}

\author[A. F. Boix]{Alberto F. Boix$^*$ $^{\ddag}$}
\thanks{$^{\ddag}$ Supported by FPU grant AP2006-00088 from Ministerio de Educaci\'on y Ciencia (Spain)}
\address{Dept. \`Algebra i Geometria,
Univ. Barcelona, Gran Via de les Corts Catalanes 585,
Barcelona 08007, SPAIN} \email{hurraca2002@yahoo.es}

\author[S. Zarzuela]{Santiago Zarzuela$^*$}
\thanks{}
\address{Dept. \`Algebra i Geometria,
Univ. Barcelona, Gran Via de les Corts Catalanes 585,
Barcelona 08007, SPAIN} \email{szarzuela@ub.edu}



\begin{abstract}

We study the generation of the Frobenius algebra of the injective hull of a complete Stanley-Reisner ring over a field with positive characteristic. In particular, by extending the ideas used by M. Katzman to give a counterexample to a question raised by G. Lyubeznik and K. E. Smith about the finite generation of Frobenius algebras, we prove that the Frobenius algebra of the injective hull of a complete Stanley-Reisner ring can be only principally generated or infinitely generated. Also, by using our explicit description of the generators of such algebra and applying the recent work by M. Blickle about Cartier algebras and generalized test ideals, we are able to show that the set of $F$-jumping numbers of generalized test ideals associated to complete Stanley-Reisner rings form a discrete subset inside the non-negative real numbers.

\end{abstract}

\maketitle

\section{Introduction}

M.~Katzman  gave in \cite{Kat10} an example of a co-finite module $M$ over a complete local ring $R$ of prime characteristic $p>0$ for which the algebra of Frobenius operators $\cF(M)$
introduced by G.~Lyubeznik and K.~E.~Smith in \cite{LS} is not finitely generated as $R$-algebra, thus giving a negative answer to a question posed in loc. cit. In his example $R$ is a non-Cohen-Macaulay quotient of a formal power series ring in three variables by a squarefree monomial ideal and $M=E_R$ is the injective hull of the residue field of $R$.

\vskip 2mm

The aim of this work is to extend his ideas to study the algebra of Frobenius operators $\cF(E_R)$ and its Matlis dual notion of algebra of Cartier operators $\cC(R)$ of any Stanley-Reisner ring $R=S/I$ defined by a squarefree monomial ideal $I\subseteq S=k[[x_1,\dots ,x_n]]$. First we obtain in Section $3$ a precise description of these algebras that shows that they can only be principally generated or infinitely generated as $R$-algebras depending on the minimal primary decomposition of the ideal. It is worthwile to mention here that when $\cF(E_R)$ is principally generated it is possible to give an algorithmic description of the tight closure of the zero submodule of $E_R$ (see \cite{Kat10b}).

\vskip 2mm

The Frobenius algebra $\cF(E_R)$ of the injective hull of the residue field of $R$ is principally generated for Gorenstein rings since in this case $E_R$ is isomorphic to the top local cohomology module $H_{\fM}^n(R)$ (see \cite{LS} for details). The converse also holds for normal varieties. In Section $4$ we explore examples of Stanley-Reisner rings at the boundary of the Gorenstein property, but we will see that one may even find both non Cohen-Macaulay examples with principally generated Frobenius algebra and Cohen-Macaulay examples with infinitely generated Frobenius algebra.

\vskip 2mm

Given a pair $(R,\fa^t)$ where $\fa \subseteq R$ is an ideal and $t\in \bR_{\geq 0}$, N. Hara and K. Y. Yoshida \cite{HaYo} introduced generalized test ideals $\tau(R,\fa^t)$ as characteristic $p>0$ analogs of multiplier ideals.
The recent development in the study of Cartier operators has given a new approach to describe generalized test ideals
of non-reduced rings and study the discretness of the associated $F$-jumping numbers (see \cite{ST11} for a recent survey on the subject). In this line of research, M.~Blickle \cite{Bli09} developed the notion of Cartier algebras and proved that when a Cartier algebra is {\it gauge bounded} then the set of $F$-jumping numbers of the corresponding generalized test ideals is discrete. He also proved that finitely generated Cartier algebras are gauge bounded.

\vskip 2mm

For a complete, local and $F$-finite ring it turns out that its algebra of Cartier operators is isomorphic to the opposite of the Frobenius algebra on the injective hull of its residue field. In particular, if $R$ is a Stanley-Reisner ring then its algebra of Cartier operators $\mathcal{C} (R)$ is an $R$-Cartier algebra which is isomorphic to the opposite of the Frobenius algebra $\cF(E_R)$. By using the description we have obtained in Section $3$ for this Frobenius algebra, we are able to prove in Section $5$ that the corresponding Cartier algebra $\mathcal{C} (R)$ is gauge bounded even in the case that it is infinitely generated. Therefore, the $F$-jumping numbers of generalized tests ideals associated to ideals of Stanley-Reisner rings form always a discrete set.

\vskip 2mm


The composition of a Cartier and a Frobenius operator gives a differential operator. This is a natural pairing that is an isomorphism for regular rings, being this fact one of the building blocks in the so-called {\it Frobenius descent} (see \cite{ABL} for an introduction). In general, this pairing is far from being an isomorphism and so it is natural to ask for its image. In the case of a Stanley-Reisner ring $R$ we are able to check out as a final application of our description of its ring of Cartier operators $\cC(R)$ how far this pairing might be from being surjective.

\section{Frobenius algebras vs. Cartier algebras}

Let $R$ be a commutative ring of characteristic $p>0$.  The aim of this Section is to introduce the basic facts on the ring of Frobenius operators $\cF(M)$ and the ring of Cartier operators $\cC(M)$ of an $R$-module $M$.
To this purpose, recall that we can use the $e$-th iterated Frobenius map $F^e:R\lra R$ to define a new $R$-module structure on $M$ given by $rm:=r^{p^e}m$. One denotes this $R$-module $F^e_{\ast}M$. In fact, we can use this to define the {\it $e$-th Frobenius functor}  $F^e_{\ast}$ from the category of left $R$-modules to itself.

\subsection{Frobenius algebras}

A \textit{$p^{e}$-linear map} $\varphi_e: M\lra M$ is an additive map  that satisfies
$\varphi_{e}(r m)=r^{p^e}\varphi_{e}(m) $ for all $r\in R$, $m \in M$.
The set of all $p^{e}$-linear maps is identified with the abelian group $$\cF^e(M):=\Hom_R(M,F_{\ast}^eM).$$
We remark the following facts:

\begin{itemize}
\item[$\cdot$] Composing a $p^{e}$-linear map  and a $p^{e'}$-linear map
 in the obvious way as additive maps we get a $p^{(e+e')}$-linear map.

\item[$\cdot$] Each $\cF^e(M)$ is a left module over $\cF^0(M):=\End_R(M)$.
\end{itemize}

\noindent This leads G.~Lyubeznik and K.~E.~Smith \cite{LS} to cook up the following ring.

\begin{definition}
The ring of Frobenius operators on $M$ is the graded, associative, not necessarily commutative  ring
$$\cF(M):=\bigoplus_{e\geq 0}\cF^e(M)$$

\end{definition}

In loc. cit. the authors raise the question whether $\cF(M)$ is finitely generated as an $\cF^0(M)$-algebra.

\vskip 2mm

Given any commutative ring $R$ one may construct the $e$-th Frobenius skew polynomial ring $R[\theta;F^e]$ as the left $R$-module freely generated by $\{\theta^i\}_{i\geq 0}$ with the multiplication subject to the rule $\theta r=r^{p^e}\theta$ for all $r\in R$. The point of view considered in \cite{LS} is that  $\cF^e(M)$ may be identified with the set of all (left) $R[\theta;F^e]$-module structures on $M$. In particular they also proved:

\begin{itemize}
 \item[$\cdot$] $\cF(R)\cong R[\theta;F].$

\item[$\cdot$] $\cF(H^n_{\fM}(R))\cong R[\theta;F],$ where $H^n_{\fM}(R)$ is the top local cohomology module of a complete $n$-dimensional local Cohen-Macaulay ring $(R,\fM)$.
\end{itemize}

More generally, given $u\in R$, one may also consider the $R$-subalgebra of $R[\theta;F^e]$ generated by the element $u\theta$. In fact, this is also the skew polynomial ring $R[u\theta;F^e]$, freely generated as left $R$-module by $\{(u \theta)^i\}_{i\geq 0}$ with the multiplication $ u\theta r=r^{p^e}u\theta$. These are the examples of principally generated Frobenius algebras that will appear in this work. For more information of skew polynomial rings and Ore extensions one may consult \cite{GW}.

\subsubsection{Frobenius algebra of injective hulls}
Let $S=k[[x_1,...,x_n]]$ be the formal power series ring in $n$
variables over a field $k$ of characteristic $p>0$. Let $I\subseteq
S$ be any ideal and  $E_R$ be the injective hull of the residue field of $R=S/I$.  For this module we have a nice description of the corresponding Frobenius algebra (see  \cite{Bli01}, \cite{Kat08}). Namely, there exists a natural Frobenius action $F$ from $E_R$ onto itself such that for each $e\geq 0$, any $p^e$-linear map from $E_R$ onto itself is uniquely of the form $gF^e$, where $g$ is an element of $(I^{[p^e]}:_S I)/ I^{[p^e]}$. So there exists an isomorphism of $R$-modules $$\cF^e(E_R) \cong  (I^{[p^e]}:_S I)/ I^{[p^e]}$$ that can be extended in a natural way to an isomorphism of $R$-algebras $$\cF(E_R)\cong \bigoplus_{e\geq 0} (I^{[p^e]}:_S I)/ I^{[p^e]}$$
Notice that $\cF(E_R)$ is an $R$-algebra due to the fact that $\cF^0(E_R)=\Hom_R(E_R,E_R)\cong R$.

\vskip 2mm

In the following particular case we may give a more precise description of the structure of a principally generated Frobenius algebra in terms of  skew polynomial rings.

\begin{lemma}
Assume that there is $u\in S$ such that for all $q=p^e$, $e\geq 0$
$$(I^{[q]}:_S I)= I^{[q]} + (u^{q-1}).$$ Then there is an isomorphism of $R$-algebras $\cF(E_R)\cong R[u^{p-1}\theta;F]$.

\end{lemma}

\subsection{Cartier algebras}

A \textit{$p^{-e}$-linear map} $\psi_e: M\lra M$ is an additive map that satisfies
$\psi_{e}(r^{p^e} m)=r\psi_{e}(m) $ for all $r\in R$, $m \in M$. We identify the set of $p^{-e}$-linear maps
with the abelian group $$\cC_e(M):=\Hom_R(F_{\ast}^eM,M)$$

Analogously to the case of Frobenius algebras we have the following facts:

\begin{itemize}
\item[$\cdot$] Composing a $p^{-e}$-linear map and a $p^{-e'}$-linear map
 in the obvious way as additive maps we get a $p^{-(e+e')}$-linear map.

\item[$\cdot$] Each $\cC_e(M)$ is a right module over $\cC_0(M):=\End_R(M)$.
\end{itemize}

Thus we may define:

\begin{definition}
The ring of Cartier operators on $M$ is the graded, associative, not necessarily commutative  ring
$$\cC(M):=\bigoplus_{e\geq 0}\cC_e(M)$$

\end{definition}

We should point out that this algebra does not generally fit into the notion of \textit{$R$-Cartier algebra} (or
$R$-algebra of Cartier type) defined by M.~Blickle in \cite[Def. $2.2$]{Bli09} because $\cC_0(M)$ might be too big.
His definition is as follows.

\begin{definition}
Let $R$ be a commutative noetherian ring containing a field of prime characteristic $p>0$. An \textit{$R$-Cartier algebra} (or $R$-algebra of Cartier type) is a positively graded $R$-algebra
$ \mathcal{C}:=\bigoplus_{e\geq 0}\mathcal{C}_e$
 such that $ r\cdot\psi_e=\psi_e\cdot r^{p^e}$ for all $r\in R,\psi_e\in \mathcal{C}_e.$
The $R$-algebra structure of $\mathcal{C}$ is nothing but a ring homomorphism $\xymatrix{R\ar[r]& \mathcal{C}}$; in fact, the image of $R$ lies in $\mathcal{C}_0$ since $1\in R$ maps to $1\in\mathcal{C}_0$. One assumes that this structural map $\xymatrix{R\ar[r]& \mathcal{C}_0}$ is surjective.
\end{definition}

\begin{example}
Let $R$ be a commutative noetherian ring containing a field of characteristic $p>0$. The algebra $\mathcal{C} (R)$ of Cartier operators on $R$ is an $R$-Cartier algebra, since $\End_R(R) = R$.
\end{example}

Given any commutative ring $R$ one may also construct the $e$-th Frobenius skew polynomial ring $R[\varepsilon;F^e]$ as the right $R$-module freely generated by $\{\varepsilon^i\}_{i\geq 0}$ with the multiplication subject to the rule $ r \varepsilon = \varepsilon r^{p^e}$ for all $r\in R$. In fact, $R[\varepsilon;F^e]$ is isomorphic to $R[\theta;F^e]^{op}$, the opposite ring of $R[\theta;F^e]$. As in the case of the Frobenius operators, the set $\cC_e(M)$ may be identified with the set of all (left) $R[\varepsilon;F^e]$-module structures on $M$, or equivalently with the set of (right) $R[\theta;F^e]$-module structures on $M$.

\vskip 2mm

Again, for any element $u\in R$, one may also consider the $R$-subalgebra of $R[\varepsilon;F^e]$ generated by the element $\varepsilon u$, which is isomorphic to the skew polynomial ring $R[\varepsilon u;F^e]$ freely generated as a right $R$-module by $\{(\varepsilon u)^i\}_{i\geq 0}$ with the multiplication $ r \varepsilon u = \varepsilon u r^{p^e}$. These are the examples of principally generated $R$-Cartier algebras that will appear in this work.

\subsubsection{Matlis duality}
Let $(R,\fM)$ be complete, local and $F$-finite. Then Matlis duality induces an equivalence of categories between left $R[\theta;F^e]$-modules which are co-finite as $R$-modules and $R$-finitely generated right $R[\theta, F^e]$-modules, equivalently $R$-finitely generated left $R[\varepsilon, F^e]$-modules, see \cite{BB}, \cite{SY} for details. This equivalence is compatible with the corresponding ring structures for the Frobenius operators and Cartier operators. It follows that the $R$-Cartier algebra $\mathcal{C} (R)$ is isomorphic to the opposite of the algebra $\cF(E_R)$ of Frobenius operators on the injective hull of the residue field of $R$.

\section{Main result}

Let $S=k[[x_1,...,x_n]]$ be the formal power series ring in $n$
variables over a field $k$ of characteristic $p>0$. Let $I\subseteq
S$ be an ideal generated by squarefree
monomials ${\bf x^{\alpha}}:= x_1^{\alpha_1}\cdots
x_n^{\alpha_n}, \hskip 2mm {\bf \alpha}\in
\{0,1\}^n$. Its minimal primary decomposition $I=I_{\alpha_1} \cap \cdots \cap I_{\alpha_s}$ is given in terms
of { face ideals} ${I_{\alpha}}:= \langle x_i\hskip 2mm | \hskip 2mm
\alpha_i \neq 0 \rangle, \hskip 2mm  {\bf \alpha}\in \{0,1\}^n.$
For simplicity we will
denote the homogeneous maximal ideal $\fM:=I_{{\bf
1}}=(x_1,\dots,x_n)$, where ${\bf 1}=(1,\dots,1)\in \{0,1\}^n$.
 The positive support of ${\alpha}$ is
${\rm supp}_{+}(\alpha):=\{i \hskip 1mm | \hskip 1mm \alpha_i>0 \}$ and we will denote $|\alpha|= \alpha_1
+\cdots+\alpha_n$.

\vskip 2mm

The main goal of this Section is describing $(I^{[q]}:_S I)$ for any squarefree monomial ideal $I\subseteq S$, where $q=p^e$, $e\geq 0$.  First
we collect some well-known general facts on colon ideals. Let $J\subset S$ be an ideal:

\vskip 2mm

\begin{itemize}

\item[$\cdot$] $(\bigcap_{i=1}^r I_{\alpha_i}:_S J)= \bigcap_{i=1}^r (I_{\alpha_i}:_S J)$

\vskip 2mm

\item[$\cdot$] $(J :_S \bigcap_{i=1}^r I_{\alpha_i})\supseteq \sum_{i=1}^r (J:_S I_{\alpha_i})$

\end{itemize}

Moreover, assume that $I=({f_1}, \cdots, {f_t})
\subseteq S$. Then:

\begin{itemize}

\item[$\cdot$] $(J:_S I)= \bigcap_{i=1}^t (J:_S ({f_i}))$

\vskip 2mm

\item[$\cdot$] $(J :_S ({f_i}))= \frac{1}{{f_i}} (J\cap ({f_i}))$

\end{itemize}

For the very particular case of face ideals we get:

\begin{lemma}
Let $I_{\alpha}, I_{\beta}\subseteq S=k[[x_1,...,x_n]]$ be face
ideals. Then:

\begin{itemize}
\item [(i)] $(I_{\alpha}^{[q]}:_S I_{\alpha})= I_{\alpha}^{[q]}+({\bf
x}^{\alpha})^{q-1} $

\vskip 2mm

\item [(ii)] $(I_{\alpha}^{[q]}:_S I_{\beta})= I_{\alpha}^{[q]} $

\end{itemize}

\end{lemma}

\begin{proof}
We only have to use the previous properties as follows:

\begin{itemize}
\item [(i)]$
(I_{\alpha}^{[q]}:_S I_{\alpha})= \bigcap_{i\in {\rm
supp}_{+}(\alpha)}(I_{\alpha}^{[q]}:_S x_i)= \bigcap_{i\in {\rm
supp}_{+}(\alpha)}(I_{\alpha}^{[q]}+
(x_i^{q-1}))=I_{\alpha}^{[q]}+({\bf x}^{\alpha})^{q-1}.$

\item [(ii)] Pick up $k\in {\rm supp}_{+}(\beta)$ such that $k \not\in
{\rm supp}_{+}(\alpha)$. Then:  $$ I_{\alpha}^{[q]}\subseteq
(I_{\alpha}^{[q]}:_S I_{\beta})= \bigcap_{i\in {\rm
supp}_{+}(\beta)}(I_{\alpha}^{[q]}:_S x_i) \subseteq
(I_{\alpha}^{[q]}:_S x_k) =  I_{\alpha}^{[q]}$$

\end{itemize}

\end{proof}

Notice that for height one face ideals we have $I_{\alpha}^{[q]}+({\bf x}^{\alpha})^{q-1} =({\bf x}^{\alpha})^{q-1} $. It is also useful
to point that if $I_{\alpha}=(x_{i_1},\dots ,x_{i_r})$ then $I_{\alpha}^{[q]}+({\bf
x}^{\alpha})^{q-1} = \bigcap_{j=1}^r(x_{i_1}^q,\dots , x_{i_j}^{q-1},\dots ,x_{i_r}^q)$.

\vskip 2mm

The main result in this Section is:

\begin{proposition}
Let $I=I_{\alpha_1} \cap \cdots \cap I_{\alpha_s}$ be the minimal
primary decomposition of a squarefree monomial ideal $I\subseteq
S$. Then:
$$(I^{[q]}:_S I)= (I_{\alpha_1}^{[q]}+({\bf
x}^{\alpha_1})^{q-1}) \cap \cdots \cap (I_{\alpha_s}^{[q]}+({\bf
x}^{\alpha_s})^{q-1}) $$

\end{proposition}

Before proving this result we need to introduce some notation. For every component $I_{\alpha_i}$ we will associate a (non unique) ideal $0\neq I_i \subseteq I$ having a very specific form. For simplicity we will just develop the case $i=1$. Let $x_j\in I_{\alpha_1}$ a minimal generator of $I_{\alpha _1}$. Then one can always find a minimal generator $x_{i_1}x_{i_2}\cdots x_{i_r}$ of $I$ such that $x_{i_1}= x_j$ and $x_{i_k}$ does not belong to $I_{\alpha _1}$ for any $k = 2, \dots ,r$. We can do so because we are considering a minimal primary decomposition of $I$. Now do this for each generator of $I_{\alpha_1}$ and define $I_1$ as the ideal generated by these monomials. Note that the monomials are all distinct.

\begin{example}
 Let $I=(xyt,xz,yz,zt)=(x,y,t)\cap(x,z)\cap(y,z) \cap (z,t) \subseteq k[[x,y,z,t]]$. We have the following ideals associated to each face ideal:

\begin{itemize}
 \item[$\cdot$] $I_1 = (xz,yz,zt)$

 \item[$\cdot$] $I_2 = (xyt,yz)$

 \item[$\cdot$] $I_3 = (xyt,xz)$

 \item[$\cdot$] $I_4 = (xz, xyt)$

\end{itemize}

\noindent
Note that we could also choose $I_2 = (xyt, zt)$, or $I_4 = (yz, xyt)$.

\end{example}

\begin{proof}[Proof of Proposition $3.2$] We have:
$$(I^{[q]}:_S I)=(I_{\alpha_1}^{[q]} \cap \cdots \cap I_{\alpha_s}^{[q]}:_S I)=
(I_{\alpha_1}^{[q]}:_S I) \cap \cdots \cap (I_{\alpha_s}^{[q]}:_S
I)$$
To obtain the inclusion $\supseteq$ in the statement notice that, for each face ideal $I_{\alpha_i}$,
$$(I_{\alpha_i}^{[q]}:_S
I)=(I_{\alpha_i}^{[q]}:_S I_{\alpha_1} \cap \cdots \cap I_{\alpha_s})\supseteq \sum_{j=1}^r (I_{\alpha_i}:_S I_{\alpha_j})=I_{\alpha_i}^{[q]}+({\bf
x}^{\alpha_i})^{q-1},$$ where the last equality comes from Lemma $3.1$. On the other direction, recall that for each $I_{\alpha_i} $ we constructed the ideal $I_i\subset I$. Then
$$(I_{\alpha_i}^{[q]}:_S I) \subseteq (I_{\alpha_i}^{[q]}:_S I_i)$$

\noindent The generators of $I_i$ of degree $r$ are $x_{i_1}x_{i_2}\cdots x_{i_r}$ satisfying $i_1\in {\rm supp}_{+}(\alpha_i), {i_2},\dots,i_r\not\in {\rm supp}_{+}(\alpha_i)$. Therefore, from the fact that
$$(I_{\alpha_i}^{[q]}:_S x_{i_1}\cdots x_{i_r})= \frac{1}{x_{i_1}\cdots x_{i_r}}(I_{\alpha_i}^{[q]}\cap (x_{i_1}\cdots x_{i_r}))=\begin{cases} I_{\alpha_i}^{[q]}+\langle x_{i_1}^{q-1}\rangle,\text{ if }
\hlt(I_{\alpha_i})>1,\\ I_{\alpha_i}^{[q-1]},\text{ if }\hlt(I_{\alpha_i})=1.\end{cases}$$ and that for each generator of $I_{\alpha_i}$ we have one generator of $I_i$ we obtain that
$$(I_{\alpha_i}^{[q]}:_S I) \subseteq (I_{\alpha_i}^{[q]}:_S I_i)= \bigcap (I_{\alpha_i}^{[q]}:_S x_{i_1}\cdots x_{i_r})=I_{\alpha_i}^{[q]}+({\bf
x}^{\alpha_i})^{q-1}$$
\end{proof}

Note that under the hypothesis of the above proposition we have in fact the equality
$$(I^{[q]}:_S I)= (I_{\alpha_1}^{[q]}:_S I_{\alpha_1}) \cap \cdots \cap (I_{\alpha_s}^{[q]}:_S
I_{\alpha_s})$$
In the unmixed case this equality formula may also be deduced from R.~Fedder \cite[Lemma 4.1]{Fed}.

\subsection{Consequences}

Let $I=I_{\alpha_1} \cap \cdots \cap I_{\alpha_s}$ be the minimal
primary decomposition of a squarefree monomial ideal $I\subseteq
S$ and $R=S/I$ the corresponding Stanley-Reisner ring. From the formula $$(I^{[q]}:_S I)=  (I_{\alpha_1}^{[q]}+({\bf
x}^{\alpha_1})^{q-1}) \cap \cdots \cap (I_{\alpha_s}^{[q]}+({\bf
x}^{\alpha_s})^{q-1})$$ we can easily deduce that
the Frobenius algebra $\cF(E_R)$ is either infinitely generated or principally generated.
If we take a close look we see that only the following situations are possible\footnote{For simplicity we will assume that $\fM=I_{\alpha_1} + \cdots + I_{\alpha_s}$.}:
\begin{itemize}

\item[(i)] If $\hlt(I_{\alpha_i}) > 1$ for all $i=1,\dots,s$ then
$$(I^{[q]}:I)= I^{[q]} + J_q + ({\bf x}^{\bf 1})^{q-1},$$ where the generators ${\bf x}^{\gamma}=x_1^{\gamma_1}\cdots
x_n^{\gamma_n}$ of $J_q$ satisfy $\gamma_i\in \{0,q-1,q\}$. Notice that ${\bf x}^{\gamma} \in ({\bf x}^{\bf 1})^{q-1}$ in the case  when $\gamma_i\neq 0$ $\forall i$ and that we may also have ${\bf x}^{\gamma} \in I^{[q]}$  when a generator ${\bf x}^{\alpha}$ of $I^{[q]}$ divides ${\bf x}^{\gamma}$. In particular we end up with only two possibilities depending whether $J_q$ is contained in $ I^{[q]} + ({\bf x}^{\bf 1})^{q-1}$ or not.

\vskip 2mm

\begin{itemize}

\item[(a)] $(I^{[q]}:I)= I^{[q]} + ({\bf x}^{\bf 1})^{q-1}$

\vskip 2mm

\item[(b)] $(I^{[q]}:I)= I^{[q]} + J_q + ({\bf x}^{\bf 1})^{q-1}$

\end{itemize}

\vskip 2mm
\noindent In the later case, there exists a generator ${\bf x}^{\gamma}$ of $J_q$ having $\gamma_i=q$, $\gamma_j=q-1$, $\gamma_k=0$ for $1\leq i,j,k \leq n$.

\vskip 2mm

\item[(ii)] If $\hlt(I) = 1 $ and there is $i\in \{1,\dots,s\}$ such that $\hlt(I_{\alpha_i}) > 1$ then
$$(I^{[q]}:I)= J'_q + ({\bf x}^{\bf 1})^{q-1}.$$ In this case, at least there exists a generator ${\bf x}^{\gamma}$ of $J'_q$ such that ${\bf x}^{\gamma} \not\in I^{[q]}+({\bf x}^{\bf 1})^{q-1}$.

\vskip 2mm

\item[(iii)] If $\hlt(I_{\alpha_i}) = 1$ for all $i=1,\dots,s$ then
$$(I^{[q]}:I)= ({\bf x}^{\bf 1})^{q-1}.$$ Notice that in this case $R=S/I$ is Gorenstein.

\end{itemize}

\vskip 2mm

\begin{proposition}
With the previous assumptions:

\begin{itemize}
\item[$\cdot$] $\cF(E_R)\cong S[({\bf x}^{\bf 1})^{(p-1)}\theta;F]$ is principally generated in cases $(i.a)$ and $(iii)$.

\item[$\cdot$]  $\cF(E_R)$ is infinitely generated in cases $(i.b)$ and $(ii)$.

\end{itemize}
\end{proposition}

\vskip 2mm

The first part follows from Lemma $2.2$. We will see that the arguments used by M.~Katzman in \cite{Kat10} can be generalized  to obtain the second part.

\subsubsection{M.~Katzman's criterion}

Let $S$ be a regular complete local ring containing a field of characteristic $p>0$. Let $I\subseteq S$ be an ideal and $E_R$ be the injective hull of the residue field of $R=S/I$. For any $e\in \bN$ denote $K_e:=(I^{[p^e]}:_R I)$ and
$$L_e:= \sum_{{\tiny
\begin{tabular}{l}
$1\leq\beta_1,\dots ,\beta_s<e$ \\
$\beta_1+\cdots+\beta_s=e$
\end{tabular}}} K_{\beta_1}K_{\beta_2}^{[{p^{\beta_1}}]}K_{\beta_3}^{[{p^{\beta_1+\beta_2}}]}\cdots K_{\beta_s}^{[{p^{\beta_1+\cdots+\beta_{s-1}}}]}$$

The key result he used is:

\begin{proposition}[\cite{Kat10}]
For any $e\geq1$, let $\cF_{<e}$ be the $R$-subalgebra of $\cF(E_R)$ generated by $\cF^0(E_R),\dots ,\cF^{e-1}(E_R)$. Then
 $$\cF_{<e} \cap \cF^e(E_R) = L_e$$

\end{proposition}

For his very particular example, that is $I=(x_1x_2,x_1x_3)$, he checked that for all $e\geq 1$, the element $x_1^{q}x_2^{q-1} \in K_e$ does not belong to $L_e$. Therefore $\cF^e(E_R)$ is not in $\cF_{<e}$.

\vskip 2mm

In the situation of cases $(i.b)$ and $(ii)$  we may always assume, without loss of generality, that there exists a generator ${\bf x}^{\gamma}\in K_e$ showing up in $ J_q$ or $J'_q$  such that $\gamma_1 = q, \gamma_2 = q-1, \gamma_3 = 0$ and so has the form ${\bf x}^{\gamma}=x_1^{q}x_2^{q-1}x_4^{\gamma_4}\cdots x_n^{\gamma_n}$ with $\gamma_i\in \{0,q-1,q\}$, $i=4,\dots,n$. The proof that ${\bf x}^{\gamma}\not \in L_e$ is almost verbatim, we include it here for completeness. $L_e$ is a sum of monomial ideals so ${\bf x}^{\gamma}\in L_e$ if and only if ${\bf x}^{\gamma}$ is in one of the summands. Fix $e\geq 1$ and $1\leq\beta_1,\dots, \beta_s<e$ such that $\beta_1+\cdots+\beta_s=e$. We will show that the ideal
$$K_{\beta_1}K_{\beta_2}^{[{p^{\beta_1}}]}K_{\beta_3}^{[{p^{\beta_1+\beta_2}}]}\cdots K_{\beta_s}^{[{p^{\beta_1+\cdots+\beta_{s-1}}}]}$$ does not contain ${\bf x}^{\gamma}$. In fact, it is enough to show that $$G_{\beta_1}G_{\beta_2}^{[{p^{\beta_1}}]}G_{\beta_3}^{[{p^{\beta_1+\beta_2}}]}\cdots G_{\beta_s}^{[{p^{\beta_1+\cdots+\beta_{s-1}}}]},$$ where $G_e:= (x_1^{q}x_2^{q-1}x_4^{\gamma_4}\cdots x_n^{\gamma_n})$ does not contain ${\bf x}^{\gamma}$. The exponent of $x_1$ in the generator of the product above is $$p^{\beta_1+(\beta_1+\beta_2)+\dots+(\beta_1+\dots+\beta_s)}>p^{\beta_1+\dots+\beta_s}=p^e$$
where the inequality follows from the fact that we must have $s>1$.

\vskip 2mm

Observe that any generator of $(I^{[q]}:_R I)$ of the form considered above provides a new generator for the Frobenius algebra.

\subsubsection{Final remarks}
In the case of an infinitely generated Frobenius algebra $\cF(E_R)$ we have some control on the number of generators of each graded piece $\cF^e(E_R)$:

\vskip 2mm

\begin{itemize}

\item [$\cdot$] The formula we obtain for $(I^{[q]}:I)$ is exactly the same $\forall q$, so we only have to compute $(I^{[p]}:I)$, i.e. the first graded piece $\cF^1(E_R)$,
to describe the whole Frobenius algebra.

    \vskip 2mm

\item [$\cdot$] Assume that  $\cF^1(E_R)$ has $\mu +1$ generators, $\mu$ generators coming from $J_p$ or $J'_p$ in cases $(i.b)$ and $(ii)$ and the other one being
$({\bf x}^{\bf 1})^{p-1}$. Then, each piece $\cF^e(E_R)$ adds up $\mu$ new generators, those coming from the corresponding $J_q$ or $J'_q$.

    \vskip 2mm

\item [$\cdot$] This number $\mu$ can be particularly large, take for example an ideal $I= I_{\alpha_1} \cap \cdots \cap I_{\alpha_s}$ with disjoint variables. Then one can check that $$\mu=\prod_{i=1}^s (|\alpha_i|+1) - \prod_{i=1}^s |\alpha_i| -1.$$


\end{itemize}

\section{Examples}

The Frobenius algebra $\cF(E_R)$ of the injective hull of the residue field of $R$ is principally generated for local Gorenstein rings \cite{LS} so, by duality, if $R$ is complete and $F$-finite the Cartier algebra $\cC(R)$ of a Gorenstein ring $R$ is also principally generated. The converse holds true for $F$-finite normal rings (see \cite{Bli09}). The case of squarefree monomial ideals we are dealing with are in general non-normal
so we will discuss examples at the boundary of the Gorenstein property. It turns out
that one may find examples that, from a combinatorial point of view, are very close to being Gorenstein with infinitely generated associated
algebras and examples of principally generated Frobenius and Cartier algebras that are not even Cohen-Macaulay.

\vskip 2mm

Let $I=I_{\alpha_1} \cap \cdots \cap I_{\alpha_s}$ 
be the minimal
primary decomposition of a squarefree monomial ideal of height $r$. The Gorenstein property of this type of ideals can be described using the pieces that appear in the well known Hochster's
formula\footnote{Hochster's formula gives a description of the Hilbert function of
the local cohomology modules $H_{\fM}^{r}(R/I)$ (see \cite{Sta96})}. It reads off as $I$ is Gorenstein if and only if

\vskip 2mm

$\dim_k[\widetilde{H}_{n-r-|\alpha|-1}({\rm link}_\alpha \Delta;
k)]=1$ for all $\alpha \in \{0,1\}^n$ such that $\alpha \geq
\alpha_i$, $i=1,\dots,s$.

\vskip 2mm

\noindent Here $\Delta$ denotes the simplicial complex on the set of
vertices $\{x_1,\dots,x_n\}$ corresponding to  $I$ via the
Stanley-Reisner correspondence and, given a face
$\sigma_{\alpha}:=\{x_i \hskip 2mm | \hskip 2mm \alpha_i=1\}\in
\Delta$, the link of $\sigma_{\alpha}$ in $\Delta$ is \hskip 2mm
$${\rm link}_{{\alpha}}\Delta:=\{\tau \in \Delta \hskip 2mm | \hskip
2mm \sigma_{\alpha} \cap \tau = \emptyset, \hskip 2mm
\sigma_{\alpha} \cup \tau \in \Delta \}.$$

\vskip 2mm

\subsection{Examples with pure height}
 Cohen-Macaulay, in particular Gorenstein rings, are unmixed so we will start considering ideals  such that all the ideals in the minimal primary decomposition have the same height. Moreover we will only consider ideals involving all the variables, i.e.  $\fM=I_{\alpha_1} + \cdots + I_{\alpha_s}$.


\vskip 2mm

In the following table we treat the cases of $3$, $4$ and $5$ variables. Depending on the height we count the number of ideals (up to relabeling) having principally generated (p.g.) Frobenius algebra, how many Gorenstein (Gor) rings we get among them and the number of ideals having infinitely generated (i.g.) Frobenius algebra. All the computations were performed with the help of {\tt CoCoA} \cite{C}:

$$
\begin{tabular}{|c|c|c|c|}\hline
 $n=3$ & p.g.  & ${\rm Gor}$ & i.g.  \\
 \hline
  $\hlt I=1$ & $1$ & $1$ & - \\
  $\hlt I=2$ & $2$ & $1$ & - \\
  $\hlt I=3$  & $1$ & $1$ & -  \\ \hline
\end{tabular}  \hskip .5cm
\begin{tabular}{|c|c|c|c|}\hline
 $n=4$ & p.g.  & ${\rm Gor}$ & i.g.  \\
 \hline
  $\hlt I=1$ & $1$ & $1$ & - \\
  $\hlt I=2$ & $4$ & $2$ & $3$ \\
  $\hlt I=3$ & $3$ & $1$ & - \\
  $\hlt I=4$ & $1$ & $1$ & -  \\ \hline
\end{tabular}
\hskip .5cm
\begin{tabular}{|c|c|c|c|}\hline
 $n=5$ & p.g.  & ${\rm Gor}$ & i.g.  \\
 \hline
  $\hlt I=1$ & $1$ & $1$ & - \\
  $\hlt I=2$ & $6$ & $2$ & $13$ \\
  $\hlt I=3$ & $12$ & $2$ & $10$ \\
  $\hlt I=4$ & $4$ & $1$ & - \\
  $\hlt I=5$ & $1$ & $1$ & -  \\ \hline
\end{tabular}
$$

\vskip 2mm

We point out that the unique ideals of pure height $1$ and $n$  are $I=(x_1\cdots x_n)$ and $I=(x_1,\dots ,x_n)$ respectively and they are Gorenstein.

\vskip 2mm

In three variables we have an example of a non-Gorenstein ring having principally generated Frobenius algebra: $$I=(x,y)\cap(x,z)\cap(y,z)=(xy,xz,yz) $$
It is not Gorenstein since the piece in Hochster's formula corresponding to
$\alpha=(1,1,1)$ has dimension two. We should point out that this example is
not even $\bQ$-Gorenstein.\footnote{We thank Anurag Singh for this
remark.} 
The canonical module $\omega_{R}$ is isomorphic to the height one
ideal  $(x+y,y+z)$ in the one dimensional ring $R$. Thus
$\omega_{R}^{(n)}=\omega_{R}^n$ $\forall n\geq 1$ and, for $n\geq 2$,
$\omega_{R}^n=(x^n,y^n,z^n)$ is not principal.

\vskip 2mm

In four variables we have three examples with infinitely generated Frobenius algebra:
\begin{itemize}

\item [$\cdot$] $I=(x,y)\cap(z,w)=(xz,xw,yz,yw)$

\item [$\cdot$] $I=(x,y)\cap(x,w)\cap(y,z)=(xy,xz,yw)$

\item [$\cdot$] $I=(x,y)\cap(x,w)\cap(y,w)\cap(z,w)=(xyz,xw,yw)$

\end{itemize}

\vskip 2mm

 The first example has disjoint variables and the corresponding colon ideal is
$$(I^{[q]}:_R I)=(x^qz^q,x^qw^q,y^qz^qy^qw^q,\underline{x^{q-1}y^{q-1}z^q},\underline{x^{q}z^{q-1}w^{q-1}},\underline{y^{q}z^{q-1}w^{q-1}},\underline{x^{q-1}y^{q-1}w^q}, (xyzw)^{q-1})$$

\vskip 2mm

 The second example is very close to being Gorenstein. The pieces corresponding to
$\alpha=(0,1,1,1)$ and $\alpha=(1,1,1,1)$ are zero and we have
$$(I^{[q]}:_R I)=(x^qy^q,x^qz^q,y^qw^q,\underline{x^{q}y^{q-1}z^{q-1}}, \underline{x^{q-1}y^{q}w^{q-1}},
(xyzw)^{q-1})$$

\vskip 2mm

The third example is very close to being Gorenstein in the sense that the piece
corresponding to $\alpha=(1,1,1,0)$ is zero and the one
corresponding to $\alpha=(1,1,0,1)$ has dimension two. We have
$$(I^{[q]}:_R I)=(x^qy^qz^q,x^qw^q,y^qw^q,\underline{x^{q-1}y^{q-1}w^q}, (xyzw)^{q-1})$$

\vskip 2mm

The examples with pure height $n-1$ that we computed have principally generated Frobenius algebra. It is easy to check out that this is a general fact.

\vskip 2mm

\begin{proposition}
Let $I\subseteq k[[x_1,...,x_n]]$ be a squarefree monomial of pure height $n-1$. Then, the Frobenius algebra $\cF(E_R)$ is principally generated.

\end{proposition}

\vskip 2mm

Another set of examples with principally generated Frobenius algebra is given by the ideals $$I_{k,n}:=\bigcap_{1\leq i_1<\cdots <i_k \leq n}(x_{i_1},\dots , x_{i_k}).$$ Their Alexander duals are also of this form, namely we have $I_{k,n}^{\vee}=I_{n-k+1,n}$. In general, having principally generated Frobenius algebra does not behave well for Alexander duality: take for instance $I=(x,y)\cap (z,w)$ and $I^{\vee} = (x,z) \cap (x,w) \cap (y,z) \cap (y,w)$.

\begin{proposition}
Let $I_{k,n}\subseteq k[[x_1,...,x_n]]$ be the squarefree monomial ideal obtained as intersection of all the face ideals of pure height $k$. Then, the Frobenius algebra $\cF(E_R)$ is principally generated.

\end{proposition}

\begin{proof}
We have a decomposition $I_{k,n}=  I_{k,n-1}\cap J_{k,n}$, where
$$J_{k,n}=(x_n)+ \sum_{|\alpha|=n-k+1} (x_1^{\alpha_1}\cdots x_{n-1}^{\alpha_{n-1}})$$
A direct computation shows that $J_{k,n}$ has principally generated Frobenius algebra.
The result follows using induction on $d$.
\end{proof}

\subsection{Examples with no pure height} We have seen that the Gorensteinness of the Stanley-Reisner ring $R$
is not the property that tackles when $\cF(E_R)$ is principally
generated. It turns out that one may even find examples of non
Cohen-Macaulay rings with principally generated Frobenius algebra.
In $4$ and $5$ variables we have the following examples:

\vskip 2mm

\begin{itemize}

\item [$\cdot$]  $I=(x, w)\cap (x, z)\cap (x, y)\cap (y,z,w)$

\item [$\cdot$]  $I=(x, u)\cap (x, w)\cap (x, z)\cap(x, y) \cap (y,z,w,u)$

\item [$\cdot$] $I=(x, z)\cap(x, w)\cap (x, u)\cap (y, z)\cap (y, w)\cap(y, u)\cap (z, w, u)$

\item [$\cdot$] $I= (x, y, u)\cap (x, y, w)\cap (x, y, z)\cap (y, z, w, u)$

\item [$\cdot$] $I= (x, y, w)\cap (x, y, u)\cap (x, z, w)\cap (x, z, u)\cap (y, z, w, u)$

\end{itemize}

\vskip 2mm

In general we may find examples in any dimension. Take for example
the families of ideals in $n$ variables

\begin{itemize}

\item [$\cdot$] $I=(x_1,\dots, x_r,x_{r+1}\cdots x_n)\cap (x_{r+1},\dots ,x_n)$

\item [$\cdot$] $I=(x_1,\dots, x_r,x_{r+1}\cdots x_{r_1}, x_{r_1+1}\cdots x_{r_2},
\dots , x_{r_t+1}\cdots x_n )\cap (x_{r+1},\dots ,x_n)$,

\vskip 2mm  \hskip .5cm where $1\leq r < r_1 < \cdots < r_t \leq n$.

\end{itemize}

\section{Applications}

\subsection{Generalized test ideals}
 M.~Blickle developed in \cite{Bli09} the notion of Cartier algebras and used it to define test modules $\tau(M,\mathcal{C})$ associated to pairs $(M, \mathcal{C})$ consisiting of an $R$-Cartier algebra $\mathcal{C}$ and a coherent left $\mathcal{C}$-module $M$. Inspired on previous work by K. Schwede \cite{Sch09} and using the techniques of M. Blickle and G. B\"{o}ckle \cite{BB}, he extended in this way the notion of the generalized test ideals $\tau(R,{\fa}^t)$ of N. Hara and K. Y. Yoshida \cite{HaYo} to a more general setting. In fact, when $R$ is $F$-finite and reduced, by taking $M = R$, $\fa$ an ideal of $R$, and $C$ the $R$-Cartier algebra $\mathcal{C}_R^{\fa^t}$ where $t\in \bR_{\geq 0}$, one gets that $\tau(R,\mathcal{C}_R^{\fa^t}) = \tau(R,{\fa}^t)$, see \cite[Prop 3.21]{Bli09}. He also extended the notion of jumping numbers and proved that when a Cartier algebra is gauge bounded then the set of jumping numbers of the corresponding test modules is discrete \cite[Cor 4.19]{Bli09} and that finitely generated Cartier algebras are gauge bounded \cite[Prop 4.9]{Bli09}.

\vskip 2mm

The aim of this Section is to prove that the Cartier algebra $\cC(R)$ associated to a Stanley-Reisner ring $R$ is gauge bounded even in the case that it is infinitely generated. To this purpose we are only going to introduce the notions that we really require. We refer to \cite{Bli09}, \cite{BSTZ}, \cite{Sch09}, \cite{Sch10} for non explained facts on generalized test ideals on singular varieties. A nice survey can be found in \cite{ST11}.

\vskip 2mm

We first recall what a gauge is. Let $S=k[x_1,\dots,x_n]$ be the polynomial ring in $n$ variables over
a perfect field $k$ (or more in general over an $F$-finite field $k$). For any given $\alpha=(\alpha_1,\dots,\alpha_n) \in\mathbb{N}_0^n$ we consider the maximun norm $||\alpha||:=\max_{j}\alpha_j.$ This norm induces an increasing filtration $\{ S_d \}_{d\in \mathbb{N}_0\cup\{-\infty\}}$ of $k$-subspaces, where  $S_d$ is the $k$-subspace of $S$ generated by monomials ${\bf x}^\alpha$ such that $||\alpha||\leq d$ and $S_{-\infty}:=0$.

\vskip 2mm

Let $M$ be an $S$-module finitely generated by $m_1,\dots,m_k$. The filtration on $S$  induces an increasing filtration on $M$ as follows:
$$ M_d:=\begin{cases} 0,\text{ if }n=-\infty,\\
S_d\cdot\langle m_1,\dots,m_k\rangle,\text{ if }d\in\mathbb{N}_0.
\end{cases}
$$
In this way, one sets
\begin{align*}
& \xymatrix{\delta_M:\ M\ar[r]& \mathbb{N}_0\cup\{-\infty\}}\\
& m\longmapsto\begin{cases} -\infty,\text{ if }m=0,\\
d,\text{ if }m\in M_d\setminus M_{d-1}.\end{cases}
\end{align*}
We say that $\delta=\delta_M$ is a \textit{gauge} for $M$.

\begin{remark}
 $S$ has a gauge $\delta_S$ induced by the generator $1$, i.e. $\delta_S$ is the degree given by grading the variable with the maximal norm.  If $R=S/I$ is a quotient ring, then the generator $1_R$ of $R$ induces a gauge $\delta_R$ which we shall call standard gauge.

\end{remark}

\begin{definition}
The $R$-Cartier algebra $\cC(R)$ is gauge bounded if for each/some gauge
$\delta$ on $R$ there exists a set $\{\psi_i \hskip 1mm | \hskip 1mm
\psi_i \in \cC_{e_i}(R)\}_{i\in I}$ which generates $\cC_{+}(R)$ as a
right $R$-module, and a constant $K$ such that for each $\psi_i$, one
has $\delta(\psi_i(r))\leq \frac{\delta(r)}{p^{e_i}}+\frac{K}{p-1}$,
$r\in R$. We then say that $\frac{K}{p-1}$ is the bound of the gauge.
\end{definition}

Our main result in this Section is the following:

\begin{proposition}
The $R$-Cartier algebra $\cC(R)$ associated to a Stanley-Reisner ring
$R$ over a perfect field is gauge bounded.
\end{proposition}

M. Blickle \cite[Remark 4.20]{Bli09} already pointed out that the example studied by M. Katzman in
\cite{Kat10} is gauge bounded.

\begin{proof}
Recall from Section $3.1$ that the generators of the Frobenius algebra $\cF(E_R)$ are the monomials $u^{p-1}:={\bf x}^{(p-1){\bf 1}}$ and $
{\bf x}^\gamma:=x_1^{\gamma_1}\cdots x_n^{\gamma_n}\in J_{p^e}$, in particular $\gamma_i\in\{0,p^e-1,p^e\}$, $\forall e\geq 1$. Denote $\supp_{a}(\gamma):=\{i \hskip 2mm | \hskip 2mm \gamma_i= a\}$ for $a= 0,p^e-1,p^e$.  The right $R$-module generators of $\cC_{e}(R)$ are the  unique additive maps
$\psi_{e,\gamma}:=\psi_{e} \circ {\bf x}^\gamma$, where

\begin{equation*}
 \psi_{e}(c{\bf x}^\alpha):=\begin{cases}c^{1/p^e} {\bf x}^{\frac{\alpha +1}{p^e}-1}:= c^{1/p^e} {x_1}^{^{\frac{\alpha_1 +1}{p^e}-1}} \cdots {x_n}^{^{\frac{\alpha_n +1}{p^e}-1}} \hskip 2mm
\text{ if }
p^e| \alpha_j+1 \hskip 2mm \forall j\\ 0,\text{ otherwise}.\end{cases}
\end{equation*}
i.e. is the map that sends ${\bf x}^{p^e-1}$ to $1$ and ${\bf x}^\alpha$ to $0$ for $0\leq \alpha\leq p^e-1$ (see \cite[Remarks $3.11$ and $4.7$]{BSTZ}). Therefore

\begin{equation*}
\psi_{e,\gamma}(c{\bf x}^\alpha)= \psi_{e}(c{\bf x}^{\alpha + \gamma}):=\begin{cases}c^{1/p^e} {\bf x}^{\frac{\alpha +\gamma +1}{p^e}-1} \hskip 2mm
\text{ if }
p^e|(\alpha_j+\gamma_j +1) \hskip 2mm \forall j\\ 0,\text{ otherwise}.\end{cases}
\end{equation*}
In particular, for $0\leq \alpha\leq p^e-1$ we have:
\begin{equation*}
\psi_{e,\gamma}(c{\bf x}^\alpha)= \begin{cases}c^{1/p^e} \prod_{\tiny i\in \supp_{p^e}(\gamma)} x_i \hskip 2mm
\text{ if }
\alpha_i= {\tiny \begin{cases} p^e-1, \hskip 2mm
\text{ if }  i \not \in \supp_{p^e-1}(\gamma) \\ 0, \hskip 8mm
\text{ if }  i \in \supp_{p^e-1}(\gamma).\end{cases}} \hskip 2mm \\ 0,\text{ otherwise}.\end{cases}
\end{equation*}

\noindent It follows that $\delta(\psi_{e,\gamma}(r))\leq \frac{\delta(r)}{p^e} + \frac{K}{p-1}$ \hskip 2mm $\forall r\in R$ for $K\geq \frac{p-1}{p^e}$ so $\cC(R)$ is gauge bounded with bound $\frac{1}{p^e(p-1)}$.
\end{proof}

\begin{remark}
It follows from \cite[Cor 4.9]{Bli09} that the finite dimensional $k$-vector space $$R_{\leq \lfloor \frac{1}{p^{e}(p-1)} +1 \rfloor}= R_{\leq 1}$$ is a nukleus for $(R,\cC(R))$ so any $F$-pure submodule of $R$ has generators in $R_{\leq 1}$. Recall that  it is proven in \cite{CV} that the generators of the test ideal $\tau(R)$ of a Stanley-Reisner ring $R$ are squarefree monomials .
\end{remark}

\begin{remark}
Consider the right $R$-module generator $\psi_{1,u}:=\psi_{1} \circ {\bf x}^{p-1}$ of $\cC_{1}(R)$. It sends $1$ to $1$ so it gives a splitting of the Frobenius. Thus, any Stanley-Reisner ring is $F$-split so it is also $F$-pure since $R$ is $F$-finite  and we recover the result in \cite{HR}.
\end{remark}

The main consequence that one derives from the fact that the Cartier algebra is gauge bounded is the discreteness of the $F$-jumping numbers of generalized test ideals associated to pairs in the sense of Hara and Yoshida.
We point out that the result we obtain is not covered in \cite{Bli09} since the Cartier algebra of a Stanley-Reisner
ring might be infinitely generated.

\begin{corollary}
Let $\fa\subseteq R$ be an ideal of a  Stanley-Reisner ring $R$ over  a perfect field $k$. Then the $F$-jumping numbers of the generalized test ideals $\tau(R,{\fa}^t)$ are a discrete set of $\bR_{\geq 0}$.
\end{corollary}

\begin{proof}

The result follows from the previous Proposition 5.3 and then by applying Proposition 4.15, Corollary 4.19 and Proposition 3.21 in \cite{Bli09}

\end{proof}

\subsection{Cartier operators vs. differential operators}

The ring of differential operators $D_R$ is the subring of End$_{\mathbb Z}R$, whose elements are defined inductively on the order (cf. \cite[\S 16.8]{GD67}). Namely, a differential operator of order $0$ is the multiplication by an element of $R$, and a differential operator of order $n$ is an additive map $\delta: R\longrightarrow R$ such that for every $r\in R$, the commutator $[\delta, r]=\delta\circ  r- r\circ \delta$ is a differential operator of order $\leq n-1$.

\vskip 2mm

In general $D_R$ is not Noetherian but one has a filtration of $R$-algebras $D_R^{(e)}:=\End_{R}(F^e_{\ast}R)$ such that $D_R = \bigcup_{e\geq 0} D_R^{(e)}$ for $F$-finite rings \cite[1.4.8a]{Yeku}. The elements of $D_R^{(e)}$  are nothing but differential operators that are linear over $R^{p^e}$.

\vskip 2mm

Notice that we have a natural map $\Phi_e: \cC_{e}(R) \otimes_R \cF^e(R) \lra D_R^{(e)}$, i.e. a natural map     $$\Phi_e: \Hom_R(F_{\ast}^eR,R) \otimes_R \Hom_R(R, F_{\ast}^eR) \lra \End_{R}(F^e_{\ast}R) $$ given by sending $\psi_e \otimes \phi_e$ to its composition $ \phi_e  \circ \psi_e$, where $\psi_e$ (resp. $\phi_e$) is a $p^{-e}$-linear map (resp. $p^{e}$-linear map). We are considering here $\Hom_R(F_{\ast}^eR,R)$ as a right $R$-module and $\Hom_R(R, F_{\ast}^eR)$ as a left $R$-module so the tensor product is just an abelian group and $\Phi_e$ a morphism of abelian groups. When $R$ is regular and $F$-finite this map is an isomorphism. This is a key ingredient to the so-called {\it Frobenius descent} that establishes an equivalence of categories between the category of $R$-modules and the category of $D_R^{(e)}$-modules using the Frobenius functor (see \cite{ABL} for a quick introduction). In the situation we  consider in this work this is no longer an isomorphism but we can control the image of $\Phi_e$.

\vskip 2mm

First we recall that for the case  $S=k[x_1,\dots,x_n]$ or $S=k[[x_1,\dots,x_n]]$,
$D_{S}$ is the ring extension of $S$ generated by the
differential operators $\partial_i^{t}:= \frac{1}{t! }\frac{d^{t}}{d x_i^{t}}$
\hskip 3mm  $i=1,\dots,n$, where $\frac{d^{t}}{d x_i^{t}}$ is the $t$-th
partial derivation with respect to $x_i$. Using the multi-graded notation
$$\partial^{\alpha}=\frac{1}{\alpha_1! }\frac{d^{\alpha_1}}{d x_1^{\alpha_1}} \cdots \frac{1}{\alpha_n!}\frac{d^{\alpha_1}}{d x_1^{\alpha_1}}$$ we have that  $D_S^{(e)}$ is the ring
extension of $S$ generated by the operators $\partial^{\alpha}$ with $\alpha_i<p^e$ $\forall i$.

\vskip 2mm
The ring of diferential operators of a quotient ring $R=S/I$ is $$D_R= D_S(I)/ID_S$$ where $D_S(I):=\{\delta \in D_S \hskip 2mm | \hskip 2mm \delta(I)\subseteq I\}$. For the case of Stanley-Reisner rings we have the following presentation given by W.~N.~Traves \cite{Tra}:

\begin{proposition}
Let $I=I_{\alpha_1} \cap \cdots \cap I_{\alpha_s} \subseteq S=k[x_1,\dots,x_n]$ be a squarefree monomial and $R=S/I$ be the corresponding Stanley-Reisner ring.
A monomial $x^{\beta} \partial^{\alpha}\in D_S$ is in $D_R$ if and only if for each face ideal $I_{\alpha_i}$ in the minimal primary decomposition of $I$ we have either $x^{\beta}\in I_{\alpha_i}$ or $x^{\alpha}\not\in I_{\alpha_i}$.
In particular, $D_R$ is generated as a $k$-algebra by $\{x^{\beta} \partial^{\alpha} \hskip 2mm | \hskip 2mm x^{\beta}\in I_{\alpha_i} \hskip 2mm {\rm or} \hskip 2mm x^{\alpha}\not\in I_{\alpha_i} \hskip 2mm \forall i\}$.
\end{proposition}

To describe the image of $\Phi_e: \cC_{e}(R) \otimes_R \cF^e(R) \lra D_R^{(e)}$ recall that the left $R$-generator of $\cF^e(R)$ is the $e$-th Frobenius map $F^e$ and the right $R$-module generators of $\cC_{e}(R)$ are:

\begin{itemize}

 \item [$\cdot$] $\psi_{e,u}:=\psi_{e} \circ {\bf x}^{(p^e-1){\bf 1}}$

 \item [$\cdot$] $\psi_{e,\gamma}:=\psi_{e} \circ {\bf x}^\gamma$ where ${\bf x}^\gamma:=x_1^{\gamma_1}\cdots x_n^{\gamma_n}\in J_{p^e}$, in particular $\gamma_i\in\{0,p^e-1,p^e\}$.

\end{itemize}

The action of the generators of $\cC_e(R)$ on any monomial was described above. Then we get the following result just comparing with the action of the corresponding differential operator:

\begin{itemize}

 \item [$\cdot$] $\Phi_e(\psi_{e,u} \otimes F^e) = \partial^{(p^e-1){\bf 1}} {\bf x}^{(p^e-1){\bf 1}}$

\item [$\cdot$] $\Phi_e(\psi_{e,\gamma}\otimes F^e)= {\bf x}^{p^e{\alpha}} {\partial}^{(p^e-1){\bf 1}} {\bf x}^{(p^e-1){\beta}}$ \hskip 2mm

\hskip 5mm where  $\alpha_i= {\tiny \begin{cases} 1 \hskip 2mm
\text{ if }  i \in \supp_{p^e}(\gamma) \\ 0 \hskip 4mm
\text{ otherwise } \end{cases}}$ and \hskip 3mm $\beta_i= {\tiny \begin{cases} 1 \hskip 2mm
\text{ if }  i \in \supp_{p^e -1}(\gamma) \\ 0 \hskip 4mm
\text{ otherwise } \end{cases}}$

\end{itemize}

\vskip 2mm

\noindent
Note that the above expression for the elements in the image of $\Phi_e$ is not the same as the one in Proposition 5.8, but applying the relations defining $D_R$ it is not hard to get it.

\vskip 2mm

Using the description in \cite{Tra} we find differential operators
in $D_R^{(e)}$ that do not belong to the image of $\Phi_e$. Take for
example $x_i\partial_i^{p^e-1}$ where $x_i \not \in I_\alpha$ for
some face ideal $I_\alpha$ in the minimal primary decomposition of
$I$.

\begin{example}
Let $R$ be the Stanley-Reisner ring associated to $I=(y)\cap (x,z)$.
The ring of differential operators $D_R$ is the $R$-algebra
generated by $\{x\partial_1^n
\partial_3^m, z\partial_1^n
\partial_3^m, y\partial_2^n \}_{n,m\geq 0}$ \cite[Example 4.3]{Tra}.

\vskip 2mm

The Cartier algebra is generated by $\{x^{p^e}y^{p^e-1},
z^{p^e}y^{p^e-1}, x^{p^e-1}y^{p^e-1}z^{p^e-1} \}_{e\geq 0}$ that
correspond via $\Phi$ to the differential operators
$$\{x^{p^e}\partial_1^{p^e-1}\partial_2^{p^e-1}\partial_3^{p^e-1}y^{p^e-1},
z^{p^e}\partial_1^{p^e-1}\partial_2^{p^e-1}\partial_3^{p^e-1}y^{p^e-1},
\partial_1^{p^e-1}\partial_2^{p^e-1}\partial_3^{p^e-1}x^{p^e-1}y^{p^e-1}z^{p^e-1} \}_{e\geq
0}$$ Notice that $x\partial_1^{p^e-1}$ does not even belong to the
$R$-algebra generated by this set.

\end{example}

\section*{Acknowledgements}
We would like to thank M.~Blickle, M.~Katzman, R.~Y.~Sharp and A.~Singh for many useful comments.

\end{document}